\documentclass[reqno]{amsart}
\usepackage{amsmath,amsthm,amssymb,amscd}

\newtheorem{theorem}{Theorem}[section]
\newtheorem{lemma}[theorem]{Lemma}
\newtheorem{proposition}[theorem]{Proposition}

\newtheorem{corollary}[theorem]{Corollary}
\theoremstyle{definition}

\theoremstyle{remark}
\newtheorem{remark}[theorem]{Remark}

\numberwithin{equation}{section}
\allowdisplaybreaks
\begin{document}
\title[On a Paley-type graph on $\mathbb{Z}_n$]
{On a Paley-type graph on $\mathbb{Z}_n$}


\author{Anwita Bhowmik}
 \address{Department of Mathematics, Indian Institute of Technology Guwahati, North Guwahati, Guwahati-781039, Assam, INDIA}
\email{anwita@iitg.ac.in}
 \author{Rupam Barman}
 \address{Department of Mathematics, Indian Institute of Technology Guwahati, North Guwahati, Guwahati-781039, Assam, INDIA}
 \email{rupam@iitg.ac.in}


\subjclass[2020]{05C30; 11T24; 11T30; 05C55}
\date{Revised: September 30, 2021: To appear at Graphs and Combinatorics.}
\keywords{Paley graphs; finite fields; Dirichlet characters; character sums}
\begin{abstract}
Let $q$ be a prime power such that $q\equiv 1\pmod{4}$. The Paley graph of order $q$ is the graph with vertex set as the finite field $\mathbb{F}_q$ and edges defined as, $ab$ is an edge if and only if $a-b$ is a non-zero square in $\mathbb{F}_q$. We attempt to construct a similar graph of order $n$, where $n\in\mathbb{N}$. For suitable $n$, we construct the graph where the vertex set is the finite commutative ring $\mathbb{Z}_n$ and edges defined as, $ab$ is an edge if and only if $a-b\equiv x^2\pmod{n}$ for some unit $x$ of $\mathbb{Z}_n$. We look at some properties of this graph. For  primes $p\equiv 1\pmod{4}$, Evans, Pulham and Sheehan computed the number of complete subgraphs of order 4 in the Paley graph. Very recently, Dawsey and McCarthy find the number of complete subgraphs of order 4 in the generalized Paley graph of order $q$. In this article, for primes $p\equiv 1\pmod{4}$ and any positive integer $\alpha$, we find the number of complete subgraphs of order 3 and 4 in our graph defined over $\mathbb{Z}_{p^{\alpha}}$.
\end{abstract}
\maketitle
\section{Introduction and statement of main results}
Let $q$ be a prime power such that $q\equiv 1\pmod{4}$. Let $\mathbb{F}_q$ denote the finite field with $q$ elements. Let $\mathbb{F}_q^{\ast}=\mathbb{F}_q\setminus \{0\}$ and let $S_2$ be its subset of squares. The Paley graph of order $q$ is the graph with vertex set $\mathbb{F}_q$ where $ab$ is an edge if and only if $a-b\in S_2$. The Paley graphs are connected, self-complementary and strongly regular with parameters $\left( q,\dfrac{q-1}{2},\dfrac{q-5}{4}, \dfrac{q-1}{4} \right)$, see for example \cite{elsawy2012paley}. Paley graphs were first introduced in $1933$ by R.E.A.C. Paley \cite{paley1933orthogonal}. Please see \cite{jones2017paley} for a nice survey on Paley graphs.
\par 
Let $G^{(n)}$ denote a graph on $n$ vertices and let $\overline{G^{(n)}}$ be its complement. 
Let $\mathcal{K}_m(G)$ denote the number of complete subgraphs of order $m$ in a graph $G$. Let $T_m(n)=\text{min}\left(\mathcal{K}_m(G^{(n)})+ \mathcal{K}_m(\overline{G^{(n)}})\right) $ where the minimum is taken over all graphs on $n$ vertices. Erd$\ddot{\text{o}}$s \cite{erdos1962number} proved that $$T_m(n)\leq \dfrac{{n\choose m} }{2^{{m\choose 2}-1}} $$ and conjectured that $\lim\limits_{n\rightarrow\infty} T_m(n)/{n\choose m}=2^{1-{m\choose 2}}$. Subsequent attempts by Goodman \cite{goodman1959sets} and Thomason \cite{thomason} generated the interest to calculate $T_m(n)$ for different $m\in\mathbb{N}$.
\par  In $1981$, Evans, Pulham and Sheehan \cite{evans1981number} gave a simple closed formula to calculate $\mathcal{K}_4(G(p))$, $G(p)$ being the Paley graph of order $p$ where $p$ is a prime and $p\equiv 1\pmod{4}$. Write $p=a^2+b^2$ where $a, b\in\mathbb{Z}$, and $a$ is even. Then 
$$\mathcal{K}_4(G(p))=\dfrac{p(p-1)((p-9)^2-4a^2)}{2^9\times 3}.$$ This work was extended by Atanasov et al. \cite {atanasov2014certain} for a prime power $q=p^n\equiv 1\pmod{4}$ when $p\equiv 1\pmod{4}$. 
\par
Some generalizations of Paley graphs have been studied too. Ananchuen and Caccetta \cite{ananchuen2001adjacency, ananchuen1993adjacency, ananchuen2006cubic} studied some properties of the cubic and quadruple graphs. Let $q=p^n$ with odd prime $p$, $n\in\mathbb{N}$. Let $S_3=\{x^3: x\in \mathbb{F}_q^{\ast}\}$ and $S_4=\{x^4: x\in \mathbb{F}_q^{\ast}\}$. For $q\equiv 1\pmod{3}$, the graph with vertex set $\mathbb{F}_q$ and edges $ab$ where $a-b\in S_3$ is called the cubic Paley graph. For $q\equiv 1\pmod{8}$, the graph with vertex set $\mathbb{F}_q$ and edges $ab$ where $a-b\in S_4$ is called the quadruple Paley graph. In 2006, Lim and Praeger \cite{lim2006generalised} generalized further: let $k (\in\mathbb{N})\geq 2$ and $q$ be a prime power such that $q\equiv 1\pmod{k}$ if $q$ is even, or $q\equiv 1 \pmod{2k}$ if $q$ is odd; then the generalized Paley graph $G_k(q)$ is the graph with vertex set $\mathbb{F}_q$ where $ab$ is an edge if $a-b$ is a $k$-th power residue. In a very recent paper, Dawsey and McCarthy \cite{dawsey2020generalized} have computed the number of complete subgraphs of order four in Lim and Praeger's graph using finite field hypergeometric functions, which in turn, generalizes the results of Evans et al. for $G_k(q)$. Wage \cite{wage2006character} constructed three other generalizations. Let $p$ be a prime and $t\in\mathbb{F}_p$ be fixed. All the three graphs, denoted by $G_t(p)$, ${G_t}^{\prime}(p)$ and $H_t(p)$, have $\mathbb{F}_p$ as their vertex set. In $G_t(p)$, $xy$ is an edge if $x-ty$ and $y-tx$ are quadratic residues mod $p$, and ${G_t}^{\prime}(p)$ has an edge $xy$ if either $x-ty$ or $y-tx$ is a quadratic residue mod $p$. Finally, the directed graph $H_t(p)$ has an edge $x\rightarrow y$ if both $x$ and $x-ty$ are quadratic residues mod $p$.
\par 
A natural question that arises is what the analogue of a Paley graph can be if the vertex set has $n$ vertices, in general, where $n$ is a natural number. The property $q\equiv 1\pmod{4}$ for a Paley graph of order $q$ ensures that $-1$ is a quadratic residue in $\mathbb{F}_q^{\ast}$, so an edge is well-defined. Let $\mathbb{Z}_n^{\ast}$ be the multiplicative group of units of the commutative ring $\mathbb{Z}_n$. As we will prove in Proposition \ref{prop1}, it turns out that $-1$ is a square in $\mathbb{Z}_n^{\ast}$ if and only if $n$ takes the form $n=2^s p_{1}^{\alpha_{1}}\cdots p_{k}^{\alpha_{k}}$, where the distinct primes $p_i$ satisfy $p_i\equiv 1\pmod{4}$ for all $i=1, \ldots, k$, and $s=0$ or $1$, excluding the cases $n=1$ or $n=2$ where the graphs are empty and trivial respectively. We now define a Paley-type graph on the commutative ring $\mathbb{Z}_n$, which we will denote by $G_n$. The graph $G_n$ is defined as the graph with vertex set $\mathbb{Z}_n$ where $ab$ is an edge if and only if $a-b\equiv x^2\pmod{n}$ for some $x\in\mathbb{Z}_n^{\ast}$. We look at some properties of this graph, especially those whose deductions involve similar approaches as done for the Paley graphs. 
In case of the Paley graph, character sums on $\mathbb{F}_p^{\ast}$ involving the unique quadratic character were evaluated. Here we do the same using Dirichlet characters modulo $n$. Although the graph is defined for more general values of $n$ as mentioned above, we restrict our attention to $n=p^\alpha$, where $p$ is a prime such that $p\equiv 1\pmod{4}$ and $\alpha\in\mathbb{N}$ so that $\mathbb{Z}_n^{\ast}$ is cyclic which guarantees that a unique quadratic character exists. We also omit the cases where $n$ is even since there cannot exist any complete subgraphs of order more than 2.
In the following theorem, for primes $p\equiv 1\pmod{4}$ and any positive integer $\alpha$, we find the number of complete subgraphs of order $3$ contained in the graph $G_{p^{\alpha}}$.
\begin{theorem}\label{t1}
	Let $p$ be a prime such that $p\equiv 1\pmod{4}$. For any positive integer $\alpha$, we have 
	 $$\mathcal{K}_3(G_{p^{\alpha}})=\dfrac{ p^{3\alpha-2}(p-1)(p-5)}{48}.$$
\end{theorem}
In the following theorem, we find the number of complete subgraphs of order $4$ contained in the graph $G_{p^{\alpha}}$ for primes $p\equiv 1\pmod{4}$.
\begin{theorem}\label{t2}
Let $p$ be a prime such that $p\equiv 1\pmod{4}$, and let $\alpha$ be a positive integer. Let $\chi$ denote the unique quadratic Dirichlet character mod $p^{\alpha}$ and let $\psi$ be a Dirichlet character mod $p^{\alpha}$ of order $4$. Let $J(\psi,\chi)=\sum\limits_{x\in\mathbb{Z}_{p^{\alpha}}}\psi(x)\chi(1-x)$ be the Jacobi sum of $\psi$ and $\chi$. Then 
$$\mathcal{K}_4(G_{p^{\alpha}})=\dfrac{p^{2\alpha-1}(p-1)[p^{2\alpha-2}\left\lbrace (p-9)^2-2p \right\rbrace+J(\psi,\chi)^2+\overline{J(\psi,\chi)}^2 ]}{1536}.$$ 	
\end{theorem}
\begin{remark}
	 If we take $\alpha=1$, then we can further simplify
	the Jacobi sum appearing in Theorem \ref{t2} and obtain the result proved in \cite{evans1981number}.
\end{remark}
\section{Preliminaries}
We first find out the values of $n$ for which the graph $G_n$ is well-defined. We exclude the cases $n=1,2$.
\begin{proposition}\label{prop1}
	Let $n\textgreater 2$ be an integer. There exists $x\in \mathbb{Z}_n^{\ast}$ such that $x^2\equiv -1 \pmod{n}$ if and only if $n=2^s p_{1}^{\alpha_{1}}\cdots p_{k}^{\alpha_{k}}$ where the distinct primes $p_i\equiv 1\pmod{4}$ for all $i=1, \ldots, k$ and $s=0$ or $1$.
\end{proposition}
\begin{proof}
	Let $n=2^s p_{1}^{\alpha_{1}} \cdots p_{k}^{\alpha_{k}}$ where $p_i\equiv 1\pmod{4}$ for all $i=1, \ldots , k$ and $s=0$ or $1$.
	Then $\mathbb{Z}_n^{\ast} \cong \mathbb{Z}_{p_1^{\alpha_1}}^{\ast} \times \cdots \times \mathbb{Z}_{p_k^{\alpha_k}}^{\ast}$ and each $\mathbb{Z}_{p_i^{\alpha_i}}^{\ast}$ is cyclic of order $p_i^{\alpha_i-1} \left( p_i-1\right)$.
	Let $a_i\pmod{p_i^{\alpha_i}}$ be an element of order  $4$ in $\mathbb{Z}_{p_i^{\alpha_i}}^{\ast}$. 
	Then $$\left(a_1 \pmod{p_1^{\alpha_1}}, \ldots, a_k \pmod{p_k^{\alpha_k}}\right)^2=\left( -1 \pmod{p_1^{\alpha_{1}}}, \ldots, -1 \pmod{p_k^{\alpha_{k}}}\right)$$ in $\mathbb{Z}_{p_1^{\alpha_1}}^{\ast} \times \cdots \times \mathbb{Z}_{p_k^{\alpha_k}}^{\ast}$. This gives an element $x$ in $\mathbb{Z}_n^{\ast}$ with the required property due to the isomorphism.
	\par Conversely, let there exist some $x \in \mathbb{Z}_n^{\ast}$ such that $x^2\equiv -1 \pmod{n}$.
	If $2^2$ divides $n$, then $x^2\equiv -1 \pmod{4}$ which is not possible. Therefore $n=2^s p_{1}^{\alpha_{1}} \cdots p_{k}^{\alpha_{k}}$ where $s=0$ or $1$, $k \geq 1$ and $p_i$'s are distinct primes other than $2$. Again, $x^2\equiv -1 \pmod{p_i}$ implies that $p_i \equiv 1 \pmod{4}$ for $i=1, 2, \ldots, k$.
\end{proof}
Our aim is to calculate the number of complete subgraphs in $G_n$, so we try to define an appropriate quadratic Dirichlet character analogous to the unique quadratic character on $\mathbb{F}_p^{\ast}$ ($p$ being a prime).
We begin with the definition of a Dirichlet character. Let $n\in\mathbb{Z}$. A function $\psi$  from $\mathbb {Z}$  to $\mathbb {C}$ is called a Dirichlet character modulo $n$ if it has the following properties:
\begin{enumerate}
	\item $\psi(1)=1$,
	\item $\psi(ab)=\psi(a)\psi(b)~~\forall a,b\in\mathbb{Z}$,
	\item $\psi(a)=\psi(b)$ if $a\equiv b\pmod{n}$,
	\item $\psi(a)=0$ if $\gcd(a,n)>1$.
\end{enumerate}
For $n\in\mathbb{N}$, let $\varepsilon$ be the trivial character mod $n$, which is defined as
\begin{align*}
\varepsilon(x)=
\begin{cases}
1, &\text{ if }\gcd(x,n)=1;\\
0, &\text{otherwise}.
\end{cases}
\end{align*}
Let $\psi$ be a character mod $n$. For $a\in \mathbb{Z}$, we define $\overline{\psi}(a):=\overline{\psi(a)}$ which becomes a Dirichlet character mod $n$. The set of Dirichlet characters mod $n$ forms a group under multiplication defined as $\psi\lambda(a):=\psi(a)\lambda(a)$ for  $a\in \mathbb{Z}$, where $\psi$ and $\lambda$ are two Dirichlet characters mod $n$. We also have $\psi \overline{\psi}=\varepsilon$ and $\sum\limits_{x\in\mathbb{Z}_n}\psi(x)=0$ if $\psi\neq\varepsilon$.
\par 
We are looking for a character $\psi$ such that
\begin{equation}\label{e1}
\psi(x)=1 \text{ if and only if } x\equiv a^2\pmod{n} \text{ for some } a\in\mathbb{Z}_n^{\ast}.
\end{equation}
We define three possible candidates for such a character. Let $a\in\mathbb{Z}_n^{\ast}.$
\begin{itemize}
	\item $\chi_1(a):= \left(\frac{a}{n}\right)$. This is the Jacobi symbol which is a Dirichlet character modulo $n$.
	\item For suitable $n$ when $a^{\frac{\phi(n)}{2}}\equiv \pm 1\pmod{n}$, we may define
	$\chi_2(a):=a^{\frac{\phi(n)}{2}}\pmod{n}$.
	\item	$\chi_3(a):=
	\begin{cases}
	1, &\text{ if }x^2\equiv a\pmod{n} \text{~has a solution};\\
	-1, &\text{otherwise}.
	\end{cases}$
\end{itemize}
\begin{proposition}
	Let	$k\geq 2$ be an integer. Let $n=2^s p_{1}^{\alpha_{1}} \cdots p_{k}^{\alpha_{k}}$ where $p_i$'s are distinct primes satisfying $p_i\equiv 1\pmod{4}$ for all $i=1, \ldots ,k$ and $s=0$ or $1$. Then, $\chi_1$ does not satisfy the condition (\ref{e1}).
\end{proposition}
\begin{proof}
	Here, $n=2^s p_{1}^{\alpha_{1}} \cdots p_{k}^{\alpha_{k}}$ where $k\geq 2$, $p_i\equiv 1\pmod{4}$ for all $i=1, 2, \ldots, k$ and $s=0$ or $1$.
	Then for $a\in\mathbb{Z}_n^{\ast},~\chi_1(a)=\left( \dfrac{a}{n}\right) =\left( \dfrac{a}{p_1}\right)^{\alpha_1}\left( \dfrac{a}{p_2}\right)^{\alpha_2}\cdots \left( \dfrac{a}{p_k}\right)^{\alpha_k}$ where $\left(\dfrac{.}{p}\right) $ is the Legendre symbol corresponding to the prime $p$. Let $R(p)$ and $NR(p)$ denote the set of quadratic residues modulo $p$ and the set of quadratic non-residues modulo $p$ respectively. For $1\leq i\leq k,$ let $\alpha_i=2\beta_i+\gamma_i,$ where $0\leq \gamma_i\leq 1.$
	It is enough to find $a\in\mathbb{Z}_n^{\ast}$ such that $\left(\dfrac{a}{p_i}\right)=-1$ for some $i\in \left\lbrace 1,2, \ldots, k \right\rbrace $ but ${\left( \dfrac{a}{p_1}\right)}^{\gamma_1}{\left( \dfrac{a}{p_2}\right)}^{\gamma_2}\cdots {\left( \dfrac{a}{p_k}\right)}^{\gamma_k}=1$.
	If some $\alpha_i$ is even then then we choose $b\in NR(p_i)$; then the system of equations $$x\equiv b\pmod{p_i} \text{ and }x\equiv 1\pmod{p_j}~ \forall j=1,\ldots ,k, j\neq i$$ has a solution which gives the desired $a\in\mathbb{Z}_n^{\ast}$. So let us assume that all $\alpha_i$'s are odd. Then $$\chi_1(a)=\left( \dfrac{a}{n}\right) =\left( \dfrac{a}{p_1}\right)\left( \dfrac{a}{p_2}\right)\cdots \left( \dfrac{a}{p_k}\right).$$ 
	 We choose $b\in NR(p_1)$ and $c\in NR(p_2)$, then the system of equations
	$$x\equiv b\pmod{p_1},~x\equiv c\pmod{p_2}\text{ and } x\equiv 1\pmod{p_j}~\forall j\neq 1,2$$ has a solution which gives the desired $a\in\mathbb{Z}_n^{\ast}$.
\end{proof}
\begin{proposition}\label{prop2.3}
	Let $n$ be a positive integer. Then, $\chi_2=\varepsilon$ whenever $\mathbb{Z}_n^{\ast}$ is not cyclic and consequently does not satisfy the condition \eqref{e1}.
\end{proposition}
\begin{proof}
	First, we note that $\chi_2$ is a character on $\mathbb{Z}_n^{\ast}$. This is because $\left( a^{\frac{\phi(n)}{2}}\right)^2=1$ for all $a \in \mathbb{Z}_n^{\ast}$, so $a^{\frac{\phi(n)}{2}}$ can be one of the $t$ elements modulo $n$ whose square is 1, say $h_1, h_2, \ldots, h_t$. Let $\chi_2(a)=h_i, \chi_2(b)=h_j.$ Then $\chi_2(ab)=h_i h_j$, so $\chi_2$ is a character on $\mathbb{Z}_n^{\ast}$.\\
	If $\mathbb{Z}_n^{\ast}$ is not cyclic, then $n$ can be one of the following forms:
	\begin{enumerate}
		\item $n=2^\alpha (\alpha\geq 3)$: In this case, $\mathbb{Z}_{2^\alpha}^{\ast}$ has order $2^{\alpha-1}$. But it is not cyclic, so any element raised to the power $2^{\alpha-2}$ is 1. Hence, $\chi_2=\varepsilon$.
		\item $n=2^\alpha p_1^{\alpha_1} p_2^{\alpha_2}\cdots p_k^{\alpha_k} (\alpha\geq 2, k\geq 1, \alpha_i\geq 1 \forall i, p_i$'s are distinct odd primes$)$: We first consider that $\alpha=2$. We know that $\mathbb{Z}_{2^2}^{\ast}$ is cyclic of order 2. Let $x, a_1, a_2, \ldots, a_k$ be generators of $\mathbb{Z}_{2^2}^{\ast}, \mathbb{Z}_{p_1^{\alpha_1}}^{\ast}, \mathbb{Z}_{p_2^{\alpha_2}}^{\ast}, \ldots, \mathbb{Z}_{p_k^{\alpha_k}}^{\ast}$ respectively. Then, the order of the element $(x, a_1, \ldots, a_k)$ is equal to 
		\begin{align*}&\text{lcm}\left\lbrace 2,p_1^{\alpha_1-1}(p_1-1), \ldots, p_k^{\alpha_k-1}(p_k-1) \right\rbrace \\
		&\leq p_1^{\alpha_1-1}(p_1-1) \cdots p_k^{\alpha_k-1}(p_k-1)=\frac{\phi(n)}{2}.
		\end{align*}
		So $a^{\frac{\phi(n)}{2}}=1$ for all $a\in \mathbb{Z}_n^{\ast}$. Hence, $\chi_2=\varepsilon$. If $\alpha\geq 3$, then $\mathbb{Z}_{2^\alpha}^{\ast}$ is not cyclic, and hence $\frac{2^{\alpha-1}}{2}$ is the maximum order of an element. Therefore, the order of an element in $\mathbb{Z}_{2^\alpha}^{\ast}\times \mathbb{Z}_{p_1^{\alpha_1}}^{\ast}\times \cdots \times \mathbb{Z}_{p_k^{\alpha_k}}^{\ast}$  is atmost $2^{\alpha-2} p_1^{\alpha_1-1}(p_1-1)\cdots p_k^{\alpha_k-1}(p_k-1)$ which is equal to $\frac{\phi(n)}{2}$. This gives $\chi_2=\varepsilon$.
		\item $n=2 p_1^{\alpha_1} p_2^{\alpha_2}\cdots p_k^{\alpha_k} (k\geq 2, \alpha_i\geq 1 \forall i, p_i$'s are distinct odd primes$)$: Following similarly as shown in the previous cases,  we find that $\chi_2=\varepsilon$.
		\item $n= p_1^{\alpha_1} p_2^{\alpha_2}\cdots p_k^{\alpha_k} (k\geq 2, \alpha_i\geq 1 \forall i, p_i$'s are distinct odd primes$)$: In this case also, it follows that $\chi_2=\varepsilon$, and the proof goes along similar lines.
	\end{enumerate}
This completes the proof of the proposition.
\end{proof}
\begin{proposition}\label{prop2.5}
	If $\mathbb{Z}_n^{\ast}$ is not cyclic, then there exist $a,b \in \mathbb{Z}_n^{\ast}$ such that $a,b,ab$ are all non-squares.
\end{proposition}
\begin{proof}
	 Since $\mathbb{Z}_n^{\ast}$ is not cyclic, so $n$ can be one of the following forms:
	\begin{enumerate}
		\item $n$ is a power of a prime: In this case, we have $n=2^{\alpha}$ where $\alpha\geq 3$. Let $a=3$ and $b=5$. From \cite[Proposition 4.2.2, p. 45]{ireland1990classical} we find that $a$ and $ab$ are non-squares in $\mathbb{Z}^{\ast}_{2^{\alpha}}$. Again, from \cite[Theorem $2'$, p. 43]{ireland1990classical} we see that $\mathbb{Z}^{\ast}_{2^{\alpha}}=\left\lbrace (-1)^x 5^y\mid x=0, 1 \text{ and } 0\leq y\leq 2^{\alpha-2}  \right\rbrace $ which implies that $5$ cannot be a square in $\mathbb{Z}^{\ast}_{2^{\alpha}}$. Consequently, $a=3$, $b=5$ and $ab=15$ are all non-squares in $\mathbb{Z}^{\ast}_{2^{\alpha}}$.
		\item $n$ is divided by atleast two distinct primes: Let $R$ and $NR$ be the subsets of $\mathbb{Z}_n^{\ast}$ of squares and non-squares, respectively. Let $x\in NR$. It is enough to show that $\exists$  $y\in NR$ such that $xy\in NR$. Suppose that there does not exist any such $y$. Then, $\{xy: y\in NR\}\subseteq R$ which yields $\phi(n)\leq 2|R|$. In each of the cases below, we prove that $\phi(n)\nleq 2|R|$.
		\begin{itemize}
			\item $2$ is a factor of $n$: We have the following two cases.
			\begin{enumerate}
				\item[(i)] $n=2^{\alpha} p_1^{\alpha_1}\cdots p_k^{\alpha_k}, k\geq 1, \alpha\geq 2$, $p_i$'s are distinct odd primes: In this case, $$|R|=\dfrac{\phi(n)}{2\times h \times 2^{k}},$$ where $h=1$ or $2$. Hence, $\phi(n)\nleq 2|R|$.
				\item[(ii)] $n=2p_1^{\alpha_1}p_2^{\alpha_2}\cdots p_k^{\alpha_k}, k\geq 2$, $p_i$'s are distinct odd primes: In this case $|R|=\dfrac{\phi(n)}{2^{k}}$, and hence $\phi(n)\nleq 2|R|$.
			\end{enumerate}
			\item $2$ is not a factor of $n$: In this case, we have $n=p_1^{\alpha_1}p_2^{\alpha_2}\cdots p_k^{\alpha_k}$, where $k\geq 2$ and $p_i$'s are distinct odd primes. We find that  $|R|=\dfrac{\phi(n)}{2^k}$, and hence $\phi(n)\nleq 2|R|$.
		\end{itemize} 
	\end{enumerate}
This completes the proof of the proposition.
	\end{proof}
\begin{corollary}
From Proposition \ref{prop2.5}, it follows that $\chi_3$ cannot be a character unless $\mathbb{Z}_n^{\ast}$ is cyclic.	
\end{corollary}
The above propositions suggest to take $\chi_3$ as the desired quadratic character and consequently take $n$ such that $\mathbb{Z}_n^{\ast}$ is cyclic, in order that $\chi_3$ be defined. Note that $\chi_2=\chi_3$ when $\mathbb{Z}_n^{\ast}$ is cyclic.
\par Now we evaluate some character sums which shall be required later on to prove the main results. If $n=p^\alpha$ where the prime $p$ satisfies $p\equiv 1\pmod{4}$, we denote by $\chi$ the unique character mod $n$ of order $2$. It is easy to see that $\chi(-1)=1$ so an edge in the graph $G_n$ is well-defined.
\begin{lemma}\label{1}
We have	${\phi(n)/2\choose i} p^i\equiv 0\pmod{p^\alpha}$ where $n=p^\alpha,1\leq i\leq {\alpha-1}$. 
\end{lemma}
\begin{proof}
	Enough to show $p^{\alpha-i}\mid {p^{\alpha-1}(\frac{p-1}{2})\choose i}$ for $1\leq i\leq {\alpha-1}$. We have
	\begin{align*}
	{p^{\alpha-1}(\frac{p-1}{2})\choose i}&=\dfrac{p^{\alpha-1}(\frac{p-1}{2})(p^{\alpha-1}(\frac{p-1}{2})-1)\cdots (p^{\alpha-1}(\frac{p-1}{2})-i+1)}{i!}\\
	&=\dfrac{p^{\alpha-i} p^{i-1}(\frac{p-1}{2})(p^{\alpha-1}(\frac{p-1}{2})-1)\cdots (p^{\alpha-1}(\frac{p-1}{2})-i+1)}{i!}.
	\end{align*}
	For an integer $x$, let  $v_p(x)$ be the highest power of $p$ dividing $x$ and $\sigma_p(x)$ be the sum of digits of base-$p$ representation of $x$.
	 By Legendre's formula, $v_p(i!)=\sum\limits_{k=1}^{\infty} \lfloor{\frac{i}{p^k}}\rfloor$ from which it follows that $v_p(i!)=\dfrac{i-\sigma_p(i)}{p-1}$.
	If possible let $p^i\mid i!$. Then, $v_p(i!)\geq i$, that is  $\dfrac{i-\sigma_p(i)}{p-1}\geq i$, which is not possible. This completes the proof of the lemma.
\end{proof}
\begin{lemma}\label{3}
Let $n=p^\alpha$, where $\alpha\geq 1$ and $p\equiv 1\pmod{4}$, and let $\chi$ be the unique character mod $n$ of order $2$. Let $x\in\mathbb{Z}_n$. Then we have	$\chi(x)=\chi(x+pk)$ for any integer $k$.
\end{lemma}
\begin{proof}
Let $x\in \mathbb{Z}_n$ and let $k$ be any integer. If $x\notin \mathbb{Z}_n^{\ast}$ then we have $\chi(x)=\chi(x+pk)=0$, and hence we obtain the required result. So, let $x\in \mathbb{Z}_n^{\ast}$. Then the result follows from the binomial expansion of $(x+p)^{\frac{\phi(n)}{2}}$ and Lemma \ref{1}.
\end{proof} 
\begin{lemma}\label{4} Let $n=p^\alpha$, where $\alpha\geq 1$ and $p\equiv 1\pmod{4}$. Let $\chi$ be the unique character mod $n$ of order $2$. Then, for $a\in\mathbb{Z}_n^{\ast}$, we have 
	$\sum\limits_{x\in \mathbb{Z}_n^{\ast}} \chi(x^2-a)= -(1+\chi(a))p^{\alpha-1}$.
\end{lemma}
\begin{proof}
	For $x\in\mathbb{Z}_n^{\ast}$, let $x^{-1}$ denote the multiplicative inverse of $x$ in $\mathbb{Z}_n^{\ast}$. We have
	\begin{align}\label{e2}
	\sum\limits_{x\in \mathbb{Z}_n^{\ast}} \chi(x^2-a)&= 2 \sum\limits_{\substack{x\in \mathbb{Z}_n^{\ast}\\x \text{ is a square}}} \chi(x-a)\notag\\
	&=\sum\limits_{x\in \mathbb{Z}_n^{\ast}} \chi(x-a)(1+\chi(x))\notag\\
	&=\sum\limits_{x\in \mathbb{Z}_n^{\ast}} \chi(x-a)+\sum\limits_{x\in \mathbb{Z}_n^{\ast}} \chi( x(x-a)).
	\end{align}
	Now,
	\begin{align}\label{e3}
	\sum\limits_{x\in \mathbb{Z}_n^{\ast}} \chi(x-a)&=\sum\limits_{\substack{x\in \mathbb{Z}_n^{\ast}\\x-a\in \mathbb{Z}_n^{\ast}}} \chi(x-a)\notag\\
	&=\sum\limits_{x-a\in \mathbb{Z}_n^{\ast}} \chi(x-a)
	-\sum\limits_{\substack{x-a\in \mathbb{Z}_n^{\ast}\\x\notin \mathbb{Z}_n^{\ast}}} \chi(x-a)\notag\\
	&=-\sum\limits_{\substack{x-a\in \mathbb{Z}_n^{\ast}\\ p|x}}\chi(x-a) \notag\\
	&=-\chi(a) p^{\alpha-1}
	\end{align}
	and
	\begin{align}\label{e4}
	\sum\limits_{x\in\mathbb{Z}_n^{\ast}}\chi(x(x-a))&=\sum\limits_{x\in\mathbb{Z}_n^{\ast}}\chi(x(x-a)x^{-2})\notag\\
	&=\sum\limits_{x\in\mathbb{Z}_n^{\ast}}\chi(1-a x^{-1})\notag\\
	&=\sum\limits_{x\in \mathbb{Z}_n^{\ast}}\chi(1-a x) \notag\\
	&=\chi(a) \sum\limits_{x\in \mathbb{Z}_n^{\ast}}\chi(a^{-1}-x)\notag\\
	&=\chi(a) \left[ \sum\limits_{x\in \mathbb{Z}_n}\chi(a^{-1}-x)-\sum\limits_{\substack{x\in\mathbb{Z}_n\\p\mid x}}\chi(a^{-1}-x)\right] \notag\\
	&=-p^{\alpha-1}.
	\end{align}
Combining \eqref{e2}, \eqref{e3} and \eqref{e4} we obtain the required result. 
\end{proof}
\begin{lemma}\label{5} Let $n=p^\alpha$, where $\alpha\geq 1$ and $p\equiv 1\pmod{4}$. Let $\chi$ be the unique character mod $n$ of order $2$. We have 
	$$|\{x\in \mathbb{Z}_{p^\alpha}: p\nmid x,1-x^2; \chi(1-x^2)=1 \}| =\dfrac{p^{\alpha-1}(p-5)}{2}.$$
\end{lemma}
\begin{proof}
	We have
	\begin{align}\label{e5}
	&|\{x\in \mathbb{Z}_{p^\alpha}: p\nmid x,1-x^2; \chi(1-x^2)=1\}|\notag\\
	&=\sum\limits_{\substack{x\in \mathbb{Z}_{p^\alpha} \\p\nmid x,1-x^2}} \frac{1+\chi(1-x^2)}{2}=\frac{1}{2}\sum\limits_{\substack{x\in \mathbb{Z}_{p^\alpha} \\p\nmid x,1-x^2}} 1 + \frac{1}{2} \sum\limits_{\substack{x\in \mathbb{Z}_{p^\alpha} \\p\nmid x,1-x^2}} \chi(1-x^2).
	\end{align} 
	Since $p\nmid x$, let $x=pm+k$ where $0<k<p$ and $m\in \mathbb{Z}$. If $p$ divides $x^2-1$, then  $p$ divides $k^2-1=(k-1)(k+1)$ which yields $k=1$ or $p-1$. Hence the number of $x \in \mathbb{Z}_{p^\alpha}$ such that $p\nmid x$ but $p\mid x^2-1$ is equal to $2p^{\alpha-1}$.
	Now,
	\begin{align*}
	\sum\limits_{\substack{x\in \mathbb{Z}_{p^\alpha} \\p\nmid x,1-x^2}} 1
	&=|\{x\in \mathbb{Z}_{p^\alpha}: p\nmid x\}|-|\{x\in \mathbb{Z}_{p^\alpha}: p\nmid x, p\mid (1-x^2)\}|\\
	&=\phi(p^\alpha)-2p^{\alpha-1}\\
	&=p^{\alpha-1}(p-3).
	\end{align*}
Using this and employing Lemma \ref{4} with $a=1$, \eqref{e5} yields the required result.	
\end{proof}
\begin{lemma}\label{6}
	Let $n=p^\alpha$, where $\alpha\geq 1$ and $p\equiv 1\pmod{4}$. Let $\chi$ be the unique character mod $n$ of order $2$. For $a, b\in \mathbb{Z}_n$, we have 
	\[\sum_{x\in\mathbb{Z}_n} \chi((x-a)(x-b))=\begin{cases}
	p^{\alpha-1}(p-1),& \mbox{ if }p\mid a,p\mid b;\\
	-p^{\alpha-1},&\mbox{ if }p\mid a,p\nmid b;\\
	-p^{\alpha-1},&\mbox{ if }p\nmid a,p\nmid b,p\nmid 1-b a^{-1};\\
	p^{\alpha-1}(p-1),&\mbox{ if }p\nmid a,p\nmid b,p\mid 1-b a^{-1}.
	\end{cases}\]
\end{lemma}
	\begin{proof}
		 We consider each of the four cases as given in the statement of the lemma. Let $R$ and $NR$ be the subsets of $\mathbb{Z}_n^{\ast}$ of squares and non-squares, respectively.\\
		\underline{Case 1}: Let $p\mid a$ and $p\mid b$. Then $$\sum\limits_{x\in \mathbb{Z}_n} \chi((x-a)(x-b))=\sum\limits_{x\in R}1+\sum\limits_{x\in NR}1=\phi(n)=p^{\alpha-1}(p-1).$$\\
		\underline{Case 2}: Let $p\mid a$ and $p\nmid b$. Then 
		\begin{align}\label{e12}
		\sum\limits_{x \in \mathbb{Z}_n} \chi((x-a)(x-b))&=\sum\limits_{x\in R} \chi(x-b)-\sum\limits_{x\in NR} \chi(x-b)\notag\\
		&=2 \sum\limits_{x\in R} \chi(x-b)-\sum\limits_{x\in \mathbb{Z}_n^{\ast}} \chi(x-b).
		\end{align}
		We find that 
		\begin{equation}\label{e13}
		2\sum\limits_{x\in R} \chi(x-b)=\sum\limits_{x\in \mathbb{Z}_n^{\ast}} \chi(x^2-b)=-(1+\chi(b))p^{\alpha-1}
		\end{equation}
		and
		\begin{equation}\label{e14}
		\sum\limits_{x\in \mathbb{Z}_n^{\ast}} \chi(x-b)=-\sum\limits_{p\mid x} \chi(x-b)=-\chi(b)p^{\alpha-1}.
		\end{equation}
		Combining \eqref{e12}, \eqref{e13} and \eqref{e14} we obtain the result.\\
		\underline{Case 3}: Let $p\nmid a$ and $p\nmid b$. Then
		\begin{align}\label{ez18}
		\sum\limits_{x\in \mathbb{Z}_n}\chi((x-a)(x-b))&=\sum\limits_{x\in \mathbb{Z}_n}\chi((ax-a)(ax-b))\notag\\
		&=\sum\limits_{x\in \mathbb{Z}_n}\chi((x-1)(x-b a^{-1}))\notag\\
		&=\sum\limits_{y\in\mathbb{Z}_n^{\ast}}\chi(y)\chi(y+1-b a^{-1}).
		\end{align}
		The last equality is obtained by using the substitution $x-1=y$. 
		If $p\mid 1-b a^{-1}$ then \eqref{ez18} becomes 
		\begin{equation}\label{e16}
		\sum\limits_{y\in\mathbb{Z}_n^{\ast}}\chi(y)\chi(y+1-b a^{-1})=\sum\limits_{ y\in\mathbb{Z}_n^{\ast}}1=\phi(n)=p^{\alpha-1}(p-1),
		\end{equation} and
		if $p\nmid 1-b a^{-1}$ then using the substitution $c=1-ba^{-1}$ we have in \eqref{ez18},
		\begin{align}
		\sum\limits_{y\in\mathbb{Z}_n^{\ast}}\chi(y)\chi(y+1-b a^{-1})&=\sum_{y\in\mathbb{Z}_n^{\ast}} \chi(1+cy^{-1})\notag\\
		&=\sum_{y\in\mathbb{Z}_n^{\ast}} \chi(1+cy)\notag\\
		&=\chi(c)\sum_{y\in\mathbb{Z}_n^{\ast}}\chi(c^{-1}+y)\notag\\
		&=\chi(c)\left[ -\sum_{p\mid y}\chi(c^{-1}+y)\right] \notag\\
		&=-p^{\alpha-1}.\notag 
		\end{align}
		This completes the proof of the lemma.
	\end{proof}
\section{Some basic properties of $G_n$}
In this section we look at some properties of the graph $G_n$. We have seen that the graph is well-defined for $n=2^sp_{1}^{\alpha_{1}}\cdots p_{k}^{\alpha_{k}}$, where the distinct primes $p_i\equiv 1\pmod{4}$ for all $i=1, \ldots, k$ and $s=0$ or $1$. So we consider these forms of $n$.
\par A graph is said to be regular if each vertex has the same degree, that is, each vertex is adjacent to the same number of vertices. Like the Paley graph, we see that $G_n$ is also regular.
\begin{proposition}\label{prop4}
	Let $n=2^sp_{1}^{\alpha_{1}}\ldots p_{k}^{\alpha_{k}}$, where the distinct primes $p_i\equiv 1\pmod{4}$ for all $i=1, \ldots , k$ and $s=0$ or $1$. Then $G_n$ is regular of degree equal to $\dfrac{\phi(n)}{2^k}$. 
\end{proposition}
\begin{proof}
	Let $R$ be the subset of $\mathbb{Z}_n^{\ast}$ of squares. Clearly, $G_n$ is regular of degree equal to $|R|$. Let $a_1,\ldots, a_t$ be the distinct elements in $\mathbb{Z}_n^{\ast}$ whose squares are equal to $1$. Then $|R|=\dfrac{\phi(n)}{t}$. We now find the number of solutions of $x^2\equiv 1 \pmod{n}$. This is equivalent to finding the solutions of $$x^2\equiv 1 \pmod{2^s},~~x^2\equiv 1 \pmod{p_1^{\alpha_{1}}}, ~~\ldots, ~~x^2\equiv 1 \pmod{p_k^{\alpha_{k}}}.$$ If $s=1$ then $x^2\equiv 1 \pmod{2^s}$ has 1 solution, and $x^2\equiv 1 \pmod{p_1^{\alpha_{1}}},\ldots,x^2\equiv 1 \pmod{p_k^{\alpha_{k}}}$ have two solutions each. So $t$ must be equal to $2^k$.
	Therefore, $|R| =\dfrac{\phi(n)}{2^k}$. This completes the proof of the result.
\end{proof}
A graph is said to be self-complementary if there exists a graph isomorphism from the graph to its complement. In the following, we prove that $G_n$ is not self-complementary unless $n$ is a prime.
\begin{proposition}
Let $n=2^sp_{1}^{\alpha_{1}}\cdots p_{k}^{\alpha_{k}}$, where the distinct primes $p_i\equiv 1\pmod{4}$ for all $i=1, \ldots , k$ and $s=0$ or $1$. Then $G_n$ is not self-complementary unless $n$ is a prime.
\end{proposition}
\begin{proof}
	The graph $G_n$ has  $\dfrac{n\dfrac{\phi(n)}{2^k}}{2}$ number of edges. But we know that a self-complementary graph must have $\dfrac{n(n-1)}{4}$ number of edges.
	So, if $G_n$ is self-complementary, we must have $\dfrac{n\dfrac{\phi(n)}{2^k}}{2}=\dfrac{n(n-1)}{4}$ which gives $\phi(n)=(n-1)2^{k-1}$, which is not possible except for the case $k=1,s=0, \alpha_{1}=1$. This completes the proof of the proposition.
\end{proof}
\begin{proposition}
Let $n=2^sp_{1}^{\alpha_{1}}\cdots p_{k}^{\alpha_{k}}$, where the distinct primes $p_i\equiv 1\pmod{4}$ for all $i=1, \ldots , k$ and $s=0$ or $1$. Then $G_n$ is not a complete graph.
\end{proposition}
\begin{proof}
	A necessary condition that $G_n$ is a complete graph is 
	$x-y$ is a square in $\mathbb{Z}_n^{\ast}$ for all distinct $x, y\in \mathbb{Z}_n$. This implies that $\left\lbrace x-y: x, y \in \mathbb{Z}_n, x\neq y\right\rbrace \subseteq R$, where $R$ is the subset of $\mathbb{Z}_n^{\ast}$ of squares. 
	Hence, $n-1\leq \dfrac{\phi(n)}{2^k}$, which is not possible. Thus $G_n$ is not complete.
\end{proof}
\begin{proposition}
	Let $n=2^sp_{1}^{\alpha_{1}}\cdots p_{k}^{\alpha_{k}}$, where the distinct primes $p_i\equiv 1\pmod{4}$ for all $i=1, \ldots , k$ and $s=0$ or $1$. Then the graph $G_n$ is connected.
\end{proposition}
\begin{proof}
	Let $x<y$ be two distinct vertices of the graph. Let $(i,j)$ denote the edge between two vertices $i$ and $j$. Then the edges $(x,x+1),(x+1,x+2), \ldots ,(y-1,y)$ form a path between $x$ and $y$. Hence, $G_n$ is connected.
\end{proof}
Since $1$ is always a quadratic residue modulo $n$, $G_n$ always contains a spanning cycle. In the following proposition we check when the graph is a cycle.
\begin{proposition}
Let $n=2^sp_{1}^{\alpha_{1}}\cdots p_{k}^{\alpha_{k}}$, where the distinct primes $p_i\equiv 1\pmod{4}$ for all $i=1, \ldots, k$ and $s=0$ or $1$. Then $G_n$ is a cycle if and only if $n=5$ or $10$.
\end{proposition}
\begin{proof}
	We observe that $G_n$ is a cycle if and only if $1$ and $-1$ are the only squares modulo $n$. Let the graph be a cycle. If $n\neq 5, 10$, then $n>10$ and $3\in \mathbb{Z}_n^{\ast}$, which means that $9$ is a square in $\mathbb{Z}_n^{\ast}$ which is neither $1$ nor $-1$. This completes the proof of the proposition.
\end{proof}
A graph is called vertex-transitive if given any two vertices in the graph, there is a graph isomorphism from one of the vertices to the other.
\begin{proposition}\label{prop3}
Let $n=2^sp_{1}^{\alpha_{1}}\cdots p_{k}^{\alpha_{k}}$, where the distinct primes $p_i\equiv 1\pmod{4}$ for all $i=1, \ldots , k$ and $s=0$ or $1$. Then	$G_n$ is a vertex-transitive graph.
\end{proposition}
\begin{proof}
Let  $a,b\in\mathbb{Z}_n^{\ast} $ and $a$ is a square modulo $n$. Then for $x\in\mathbb{Z}_n^{\ast}$,	$x \mapsto ax+b$ is a graph isomorphism on $G_n$. Using this isomorphism and taking appropriate values of $a$ and $b$ we see that $G_n$ becomes vertex-transitive.
\end{proof}
Let $G=(V,E)$ be a graph. If $\left\lbrace G_i \right\rbrace_{i\in I} $ is a family of edge-disjoint subgraphs of a graph $G$ such that $E(G)=\cup_{i\in I}E(G_i)$, we write $G=\bigoplus_{i\in I}G_i.$ In this case if $G_i\cong H$ for every $i\in I$, then we write $G=\bigoplus_{i\in I}H.$ We follow the notation given in \cite[p. 2855]{ashrafi2010unit}.
	Since our main focus shall be on $G_n$ for the case $n=p^\alpha$  where $p$ is an odd prime, $p\equiv 1\pmod{4}$ and $\alpha\geq 1$, we prove the following result about the structure of the graph $G_n$.
	\begin{proposition}
		Let $\alpha$ be a positive integer. Let $n=p^\alpha$, where $p$ is a prime satisfying $p\equiv 1\pmod{4}$. Let $\chi$ be the quadratic character mod $n$. Let $G(p)$ be the Paley graph of order $p$. Then a copy of $G(p)$ exists in the graph $G_n$. In fact, we can write $$G_n=\bigoplus_{i=1}^{p^{\alpha-1}} G(p)\bigoplus \left(\bigoplus_{j=1}^{\frac{p^{\alpha-1}(p^{\alpha-1}-1)}{2} }\left(\bigoplus_{m=1}^{p} K_{1,\frac{p-1}{2}} \right) \right),$$ where $K_{1,\frac{p-1}{2}}$ is a complete bipartite graph.	
	\end{proposition}
	\begin{proof}
		Let $k$ be an integer such that $0\leq k\leq p^{\alpha-1}-1$. We define a map from $G(p)$ to the subgraph of $G_n$ induced by the vertices  $\{ kp,kp+1, \ldots, kp+p-1\} $. For $i\in \mathbb{F}_p=\{ 0,1, \ldots, p-1\}$, let $i\longmapsto kp+i$. We show that this is an isomorphism. Let $i$ and $j$ be elements of $\mathbb{F}_p$ such that $i-j$ is a quadratic residue in $\mathbb{F}_p^{\ast}$. Then $(i-j)^{\frac{p-1}{2}} \equiv 1 \pmod{p}$ and so $(i-j)^{\frac{p-1}{2}}=1+p\ell$ for some $\ell\in\mathbb{Z}$. Therefore,
		$ (i-j)^{\frac{p-1}{2}\times p^{\alpha-1}}=(1+p\ell)^{p^{\alpha-1}}\equiv 1 \pmod{p^\alpha}$ by using binomial theorem and proceeding in the same way as in Lemma \ref{1} after replacing $\frac{\phi(n)}{2}$ by $p^{\alpha-1}$ in the statement of the lemma. So, $\chi((kp+i)-(kp+j))=\chi(i-j)=1$ and hence $(kp+i)-(kp+j)$ is a square modulo $p^\alpha$ and thus there is an edge between the vertices $kp+i$ and $kp+j$ in the induced subgraph of $G_n$. For the converse, let $kp+i$ and $kp+j$ be two vertices in the induced subgraph of $G_n$ connected by an edge, where $0\leq i,j\leq p-1$. Then $i-j\equiv x^2\pmod{p^\alpha}$ for some $1\leq x\leq p-1$. This implies that $i-j\equiv x^2\pmod{p}$, and hence there is an edge between the vertices $i$ and $j$ in $G(p)$.
		\par So each induced subgraph of $G_n$ comprising of the vertices $\{ kp,kp+1, \ldots, kp+p-1\} $ where $0\leq k\leq p^{\alpha-1}-1$, is a copy of $G(p)$. 
		Each such induced subgraph has $\dfrac{p(p-1)}{4}$ edges, and there are $p^{\alpha-1}$ such induced subgraphs. The total number of edges of $G_n$ exhausted this way is $p^{\alpha-1}\dfrac{p(p-1)}{4}.$ The number of edges remaining in $G_n$ is $$\dfrac{n\phi(n)}{4}-p^{\alpha-1}\dfrac{p(p-1)}{4}=p^\alpha\left( \dfrac{p-1}{4}\right) (p^{\alpha-1}-1).$$ We see how these edges are exhuasted as follows. Let $r_1,r_2,\ldots,r_{\frac{p-1}{2}}$ be the quadratic residues in the range $0$ to $p-1$. Then for $1 \leq k \leq p^{\alpha-1}-1,$
		$$\chi(r_1+kp)=1,~\chi(r_2+kp)=1,\ldots,\chi(r_{\frac{p-1}{2}}+kp)=1.$$
		So, $r_1+kp,r_2+kp,\ldots,r_{\frac{p-1}{2}}+kp$ are quadratic residues too. In fact these are the quadratic residues in the range $kp$ to $kp+p-1$. So, from the induced subgraph of $\{ 0,1,\ldots,p-1 \} $ we have the following edges: 
		\begin{align*}
		&0 \text{~~has edges with the vertices~~} r_1+kp,r_2+kp,\ldots,r_{\frac{p-1}{2}}+kp;\\
		&1 \text{~~has edges with the~~ } r_1+kp+1,r_2+kp+1,\ldots,r_{\frac{p-1}{2}}+kp+1;\\
		&\vdots\\
		&p-1 \text{~~has edges with the vertices~~} r_1+kp+p-1,r_2+kp+p-1,\ldots,r_{\frac{p-1}{2}}+kp+p-1.
		\end{align*} 
		Since each edge is counted twice, the total number of edges among these $p^{\alpha-1}$ induced subgraphs is $\dfrac{\dfrac{p(p-1)}{2}(p^{\alpha-1}-1)p^{\alpha-1}}{2}$. This completes the proof.	
	\end{proof}
\section{Proof of Theorem \ref{t1} and Theorem \ref{t2}}
We first prove some lemmas which will be required to prove the main results.
\begin{lemma}\label{9}
For a positive integer $\alpha$, let $n=p^\alpha$, where $p$ is a prime satisfying $p\equiv 1\pmod{4}$. Let $\chi$ be the quadratic character mod $n$. 
Let 
$$S=\sum\limits_{x\in\mathbb{Z}_n} \sum\limits_{y\in\mathbb{Z}_n} \chi((1-x^2)(1-y^2) (x^2-y^2))$$
and 
$$S_0=\mathop{\sum\limits \sum\limits}_{(x,y)\in X} \chi((1-x^2)(1-y^2)(x^2-y^2))$$ where 
$X=\{(x,y)\in \mathbb{Z}_n \times \mathbb{Z}_n : p \nmid x,y,1-x^2,1-y^2, x^2-y^2\}$. Then $S_0=S+4p^{2\alpha-2}$.
\end{lemma}
\begin{proof}
	We have
	$$ S=\sum_{x\in\mathbb{Z}_n} \sum_{y\in\mathbb{Z}_n} \chi((1-x^2)(1-y^2) (x^2-y^2)).$$
	Breaking the sum $\sum_{x\in\mathbb{Z}_n} \sum_{y\in\mathbb{Z}_n}$ as
 \begin{align*}
	\mathop{\sum\limits \sum\limits}_{(x,y)\in X}+ \sum_{p\mid x}\sum_y+ \sum_{\substack{p\nmid x\\p\mid 1-x^2}}\sum_y+ \sum_{\substack{p\nmid x\\p\nmid 1-x^2}}\sum_{p\mid y}+ \sum_{\substack{p\nmid x\\p\nmid 1-x^2}}\sum_{\substack{p\nmid y\\p\mid 1-y^2}}+ \mathop{\sum_{\substack{p\nmid x\\p\nmid 1-x^2}} \sum_{\substack{p\nmid y\\p\nmid 1-y^2}}}_{\substack{p\mid x^2-y^2}},
	\end{align*} we write
\begin{align}\label{e40}
	S&=S_0+T_1+T_2+T_3+T_4+T_5.
	\end{align}
	We use Lemma \ref{3} and employ Lemma \ref{4} with $a=1$ to evaluate the $T_i$'s. We have
	\begin{align}\label{e18}
	T_1 &=\sum\limits_{p\mid x}\sum\limits_y \chi(1-x^2)\chi((1-y^2)(x^2-y^2))\notag\\
	&=\sum_{t=1}^{p^{\alpha-1}}\sum_y \chi((1-y^2)(p^2 t^2-y^2))\notag \\
	&=\sum_{t=1}^{p^{\alpha-1}}\sum_{p\mid y}\chi(1-y^2)\chi (p^2 t^2-y^2)~+~\sum_{t=1}^{p^{\alpha-1}}\sum_{p\nmid y} \chi(1-y^2)\chi (p^2 t^2-y^2)\notag \\
	&=\sum_{t=1}^{p^{\alpha-1}}\sum_{p\nmid y} \chi(1-y^2)\notag \\
	&=-2p^{2\alpha-2}.
	\end{align}
	Also,
	\begin{align}\label{e19}
	T_3 &=\sum\limits_{\substack{p\nmid x\\p\nmid 1-x^2}}\sum\limits_{p\mid y}\chi(1-x^2) \chi(1-y^2)\chi(x^2-y^2)\notag\\
	&=\sum\limits_{\substack{p\nmid x\\p\nmid 1-x^2}}\sum_{t=1}^{p^{\alpha-1}}\chi(1-x^2)\chi(x^2-p^2 t^2)\notag \\
	&=\sum\limits_{\substack{p\nmid x\\p\nmid 1-x^2}}\sum_{t=1}^{p^{\alpha-1}}\chi(1-x^2)\notag\\
	&=-2p^{2\alpha-2}.
	\end{align}
	It is easy to see that 
	\begin{equation}\label{e20}
	T_2=T_4=T_5=0.
	\end{equation}
	Using \eqref{e18}, \eqref{e19} and \eqref{e20}, \eqref{e40} yields the required result.
\end{proof}
\begin{lemma}\label{10}
 For a positive integer $\alpha$, let $n=p^\alpha$, where $p$ is a prime satisfying $p\equiv 1\pmod{4}$. Let $\chi$ be the quadratic character mod $n$ and let $\psi$ be a character mod $n$ of order $4$; let $J(\psi, \chi)=\sum\limits_{x\in \mathbb{Z}_n} \psi(x) \chi(1-x)$ be the Jacobi sum of $\psi$ and $\chi$. If $$K=\sum_{x\in\mathbb{Z}_n^{\ast}} \sum_{y\in\mathbb{Z}_n^{\ast}}\chi((1-x)(1-y)(y-x)xy),$$
	then $K=J(\psi,\chi)^2+\overline{J(\psi,\chi)}^2.$
\end{lemma}
\begin{proof}
	The proof goes along similar lines as in \cite{evans1981number}. 
	For $x\in\mathbb{Z}_n^{\ast}$, we use the notations $\frac{1}{x}$ or $x^{-1}$ to refer to the multiplicative inverse of $x$ in $\mathbb{Z}_n^{\ast}$. We have
	\begin{align*}
	K&=\mathop{\sum_{\substack{p\nmid x, 1-x}}~~\sum_{\substack{p\nmid y, 1-y}} }_{p\nmid x-y}\chi((1-x)(1-y)(y-x)xyx^{-2}y^{-2})\\
	&=\mathop{\sum_{\substack{p\nmid x, x-1}}~~\sum_{\substack{p\nmid y, y-1}} }_{p\nmid x-y}\chi\left( \left( (x-1)y^{-1}\right) \left( (y-1)x^{-1}\right) (y-x)\right).
	\end{align*}
	We break the sum into two parts. One part deals with the case when $p\mid x+y-1$ and the other part deals with the case when $p\nmid x+y-1$. First we evaluate the part when $p\mid x+y-1$. Let $x+y-1=tp$ for some $k\in\mathbb{Z}, 1\leq t\leq p^{\alpha-1}$. Then $\chi((x-1)y^{-1})=\chi(-1+tpy^{-1})=\chi(-1)=1$ by Lemma \ref{3}. Similarly $\chi((y-1)x^{-1})=1$ and $\chi(y-x)=\chi(2x-1)$. Using these, we get
	\begin{align*}
	 &\mathop{\sum_{\substack{p\nmid x,x-1}}~~\sum_{p\nmid y,y-1}}_{\substack{p\nmid x-y\\p\mid x+y-1}} \chi((x-1)y^{-1})\chi((y-1)x^{-1})\chi(y-x)\\
	 &=\sum_{\substack{p\nmid x,x-1,2x-1}}\sum_{t=1}^{p^{\alpha-1}}\chi(2x-1)\\
	&=p^{\alpha-1}\sum_{\substack{p\nmid x,x-1,2x-1}}\chi(2x-1)\\
	&=p^{\alpha-1}\left[ 0-\sum_{p\mid x}\chi(2x-1)-\sum_{\substack{p\nmid x\\p\mid x-1}}\chi(2x-1)\right] \\
	&=-2 p^{2\alpha-2},
	\end{align*}
	where we have used Lemma \ref{3} as required. So $K$ is reduced to the following expression.
	\begin{equation}\label{e6}
	K=-2 p^{2\alpha-2}+\mathop{\sum_{\substack{p\nmid x,x-1}} ~~\sum_{p\nmid y,y-1}}_{\substack{p\nmid x-y \\p\nmid x+y-1}} \chi\left( \frac{x-1}{y}\right) \chi\left( \frac{y-1}{x}\right) \chi(y-x).
	\end{equation}
	Now, we use the substitution $t=\frac{x-1}{y}$ and $u=\frac{y-1}{x}$ in the above sum. Let $$U= \mathop{\sum_{\substack{p\nmid x,x-1}} ~~\sum_{p\nmid y,y-1}}_{\substack{p\nmid x-y \\p\nmid x+y-1}} \chi\left( \frac{x-1}{y}\right) \chi\left( \frac{y-1}{x}\right) \chi(y-x)$$ and $$V=\mathop{\sum_{\substack{p\nmid t,t+1}} ~~\sum_{p\nmid u,u+1}}_{\substack{p\nmid u-t \\p\nmid ut-1}}\chi(tu(u-t)(ut-1)).$$ Then we show that $U=V$ in what follows, which will imply that the substitution $t=\frac{x-1}{y}$ and $u=\frac{y-1}{x}$ is possible.  
	It is easy to see that 
	\begin{align*}
	&\left\lbrace \left( \dfrac{x-1}{y}, \dfrac{y-1}{x}\right):(x,y)\in \mathbb{Z}_n\times\mathbb{Z}_n \text{ and } p \nmid x,x-1,y,y-1,x-y, x+y-1 \right\rbrace \\
	=&\{(t,u):(t,u)\in \mathbb{Z}_n\times\mathbb{Z}_n \text{ and } p \nmid t,t+1,u,u+1,u-t, ut-1\}.
	\end{align*}Let 
	\begin{align*}
	U_1&=\{ (x,y)\in\mathbb{Z}_n\times\mathbb{Z}_n\mid p \nmid x,x-1,y,y-1,x-y, x+y-1\};\\ 
	V_1&=\{ (t,u)\in\mathbb{Z}_n\times\mathbb{Z}_n\mid p \nmid t,t+1,u,u+1,u-t, ut-1\} .
	\end{align*}
	Then $$U=\mathop{\sum\sum}_{(x,y)\in U_1}\chi\left( \frac{x-1}{y}\right) \chi\left( \frac{y-1}{x}\right) \chi(y-x) ~\text{ and }~ V=\mathop{\sum\sum}_{(t,u)\in V_1}\chi(tu(u-t)(ut-1)).$$
	We note that $|U_1|=|V_1|=(p-3)^2p^{2\alpha-2}.$ Let us define an equivalence relation on $U_1$ as $$(x,y)\thicksim(x',y') ~~\text{ if and only if }~~ x\equiv x'\pmod{p} ~\text{ and }~y\equiv y'\pmod{p}$$
	and similarly we define an equivalence relation on $V_1$ as  $$(t,u)\thicksim_1(t',u') ~~\text{ if and only if }~~ t\equiv t'\pmod{p}~\text{ and }~u\equiv u'\pmod{p}.$$
	Then 
	\begin{equation}\label{e11}
	U=\mathop{\sum\sum}_{(x,y)\in U_1}(p^{\alpha-1})^2 ~\chi\left( \frac{x-1}{y}\right) \chi\left( \frac{y-1}{x}\right) \chi(y-x),
	\end{equation}
where the summation is over distinct equivalence class representatives corresponding to the equivalence relation $\thicksim$, because each equivalence class contains $(p^{\alpha-1})^2$ elements and for $(x,y)$ and $(x',y')$ in the same equivalence class, $$\chi\left( \frac{x-1}{y}\right) \chi\left( \frac{y-1}{x}\right) \chi(y-x) =\chi\left( \frac{x'-1}{y'}\right) \chi\left( \frac{y'-1}{x'}\right) \chi(y'-x') .$$
	Again, let $(t,u)\in V_1$. Then for any $(t',u')\in V_1$ which belongs to  the equivalence class of $(t,u)$ corresponding to the equivalence relation $\thicksim_1$, $\chi(tu(u-t)(ut-1))=\chi(t'u'(u'-t')(u't'-1)).$ So
	\begin{equation}\label{e41}
	V=\mathop{\sum\sum}_{(t,u)\in V_1}(p^{\alpha-1})^2 ~\chi(tu(u-t)(ut-1)),
	\end{equation}  where the summation is over distinct equivalence class representatives corresponding to the equivalence relation $\thicksim_1$, similar to  \eqref{e11}. 
	\par Now, for a class representative $(x,y)\in U_1$ corresponding to the equivalence relation $\thicksim$, $\chi\left( \frac{x-1}{y}\right) \chi\left( \frac{y-1}{x}\right) \chi(y-x)=\chi(tu(u-t)(ut-1))$ for some $(t,u)\in V_1.$ So 
	\begin{equation}\label{e42}
	(p^{\alpha-1})^2 ~\chi\left( \frac{x-1}{y}\right) \chi\left( \frac{y-1}{x}\right) \chi(y-x) =(p^{\alpha-1})^2 \chi(tu(u-t)(ut-1)).
	\end{equation}
	 What remains to be seen is that if $(x,y)$ and $(x_1,y_1)\in U_1$ are in different equivalence classes of the equivalence relation $\thicksim$, $(t,u)$ and $(t_1,u_1)\in V_1$ are in different equivalence classes of the equivalence relation $\thicksim_1$ as well, where $t=\frac{x-1}{y},u=\frac{y-1}{x},t_1=\frac{x_1-1}{y_1},u_1=\frac{y_1-1}{x_1}.$ But this is immediate since $(x,y)\not\thicksim (x_1,y_1)$  implies that $(t,u)\not\thicksim (t_1,u_1)$. Thus combining \eqref{e11}, \eqref{e41} and \eqref{e42} we have proved that $U=V$.
	So we have from \eqref{e6}, 
	\begin{align}\label{e7}
	K&=-2p^{2\alpha-2}+\mathop{\sum_{\substack{p\nmid t,t+1}} ~~\sum_{p\nmid u,u+1}}_{\substack{p\nmid u-t \\p\nmid ut-1}} \chi(tu(u-t)(ut-1))\notag \\
	&=\mathop{\sum_{\substack{p\nmid t}}\sum_{p\nmid u}}_{\substack{p\nmid u-t \\p\nmid ut-1}} \chi(tu(u-t)(ut-1))\notag \\
	&=\sum_{\substack{p\nmid t,t-1}}\chi(t(t-1))\sum_{p\nmid u,u^2-t} \chi(u(u^2-t)),
	\end{align}
	where the last summation is obtained by the substitution $t\mapsto \frac{t}{u}$. The sum in \eqref{e7} indexed by $u$ remains to be reduced, and following similar approach as given in \cite{evans1981number}, we use Theorem 2.7 of \cite{berndt1979sums}. Then for each $t$ in the sum indexed by $t$ in \eqref{e7}, we have
	\begin{align}\label{e21}
	\sum_{p\nmid u,u^2-t} \chi(u(u^2-t))
	&=\sum_{p\nmid u,u^2-t} \psi^2(u)\psi^2(u^2-t)\notag \\
	&=\sum_{u\in\mathbb{Z}_n^{\ast}} \psi(u^2)\psi^2(u^2-t)\notag \\
	&=\sum_{u\in\mathbb{Z}_n^{\ast}} \psi(u)\psi^2(u-t)(1+\chi(u))\notag \\&=\sum_{u\in\mathbb{Z}_n^{\ast}} \psi(u)\psi^2(u-t)+\sum_{u\in\mathbb{Z}_n^{\ast}} \psi\chi(u)\psi^2(u-t).
	\end{align}
	The first sum in the above is 
	\begin{align}\label{e22}
		\sum\limits_{u\in\mathbb{Z}_n^{\ast}} \psi(u)\psi^2(u-t)&=\sum\limits_{u\in\mathbb{Z}_n^{\ast}} \psi(ut)\psi^2(ut-t)\notag \\
		&=\overline{\psi}(t)\sum\limits_{u\in\mathbb{Z}_n^{\ast}}\psi(u)\psi^2(u-1)\notag \\
		&=\overline{\psi}(t) J(\psi,\psi^2)\notag\\
		&=\overline{\psi}(t) J(\psi,\chi).
		\end{align}
	Similarly we simplify the second sum in \eqref{e21} and obtain 
	\begin{align}\label{new-1}
	\sum_{u\in\mathbb{Z}_n^{\ast}} \psi\chi(u)\psi^2(u-t)=\psi(t)J(\overline{\psi}, \chi).
	\end{align} 
	Combining \eqref{e21}, \eqref{e22}, \eqref{new-1}, and \eqref{e7}, we have
	$$K=\sum_{p\nmid t,t-1}\chi(t(t-1))[\bar{\psi}(t)J(\psi,\chi)+\psi(t)J(\overline{\psi},\chi)]$$ and hence
	$ K=J(\psi,\chi)^2+\overline{J(\psi,\chi)}^2$. This completes the proof of the lemma.
\end{proof}
Finally we prove our main results for $n=p^\alpha$ where $p$ is a prime such that $p\equiv 1\pmod{4}$ and $\alpha\in\mathbb{N}$. Since $\mathbb{Z}_{p^\alpha}^{\ast}$ is cyclic, there exists a unique character modulo $p^\alpha$ of order $2$; let us call it $\chi$. Let $\psi$ be a character mod $p^\alpha$ of order $4$. We already know from Proposition \ref{prop3} that $G_{p^\alpha}$ is vertex-transitive. Let $H$ be the subgraph of $G_{p^\alpha}$ induced by the set of squares in $\mathbb{Z}_{p^\alpha}^{\ast}$. Then $H$ is also vertex transitive. 	
\par It is known from Proposition \ref{prop4} that $G_{p^\alpha}$ is regular; we see that the same can be deduced using the character sum given below. Let $a\in\mathbb{Z}_{p^\alpha}^{\ast},$ then
$$\text{deg}(a)=\sum\limits_{\substack{b=0\vspace{0.1cm}\\ a-b\in \mathbb{Z}_{p^\alpha}^{\ast}}}^{n-1} \dfrac{1+\chi(a-b)}{2},$$ where $\text{deg}(a)$ denotes the degree of the vertex $a$ in $G_{p^\alpha}$. 
It is easy to calculate this sum and see that $\text{deg}(a)=\dfrac{\phi(p^\alpha)}{2}$ which agrees with the proposition.
\begin{proof}[Proof of Theorem \ref{t1}]
Using the property of vertex transitivity for the graphs $G_{p^\alpha}$ and $H$, we see that 
\begin{align}\label{e23}
\mathcal{K}_3(G_{p^\alpha})&= \frac{1}{3} p^\alpha\times \text{ number of triangles in }G_{p^\alpha} \text{ containing }0\notag\\
&= \frac{p^\alpha}{3}\times \text{ number of edges in } H
\end{align}
 and
 \begin{align}\label{e24}
 \text{ number of edges in } H=\frac{1}{2}\times \frac{\phi(p^\alpha)}{2}\times \text{ number of edges in }H \text{ containing }1.
 \end{align} 
 So, combining \eqref{e23} and \eqref{e24} we find that 
\begin{align}\label{e8}
\mathcal{K}_3(G_{p^\alpha})&=\dfrac{p^\alpha\phi(p^\alpha)}{12}\times \text{ number of edges in } H \text{ containing } 1\notag \\
&=\dfrac{p^\alpha\phi(p^\alpha)}{12} \times \sum\limits_{\substack{x\in \mathbb{Z}_{p^\alpha}\\x,x-1\in\mathbb{Z}_{p^\alpha}^{\ast}}} \left(\frac{1+\chi(x)}{2}\right)\left(\frac{1+\chi(1-x)}{2}\right).
\end{align}
Lemma \ref{3} yields
\begin{equation}\label{e25}
\sum\limits_{\substack{x\in \mathbb{Z}_{p^\alpha}\\x,x-1\in\mathbb{Z}_{p^\alpha}^{\ast}}} 1 =p^{\alpha -1}(p-2)
\end{equation}
and
\begin{align}\label{e26}
\sum\limits_{\substack{x\in \mathbb{Z}_{p^\alpha}\\x,x-1\in\mathbb{Z}_{p^\alpha}^{\ast}}} \chi(x)&=\sum\limits_{x\in \mathbb{Z}_{p^\alpha}^{\ast}} \chi(x)-\sum\limits_{\substack{x\in \mathbb{Z}_{p^\alpha}^{\ast}\\p\mid x-1}}\chi(x)\notag\\
&=-\left[\chi(1)+ \chi(p+1)+\chi(2p+1)+\cdots +\chi((p^{\alpha-1}-1)p+1)\right]\notag \\
&=-p^{\alpha-1}.
\end{align}
Similarly, we have
\begin{equation}\label{e27}
\sum\limits_{\substack{x\in\mathbb{Z}_{p^\alpha}\\x,x-1\in\mathbb{Z}_{p^\alpha}^{\ast}}} \chi(1-x)=-p^{\alpha-1} 
\end{equation} and
\begin{align}\label{ee28}
\sum\limits_{\substack{x\in \mathbb{Z}_{p^\alpha}\\x,x-1\in\mathbb{Z}_{p^\alpha}^{\ast}}}\chi(x(x-1)) 
&=\sum\limits_{p\nmid x}\chi(x(x-1))\chi(x^{-2})\notag\\ 
&=\sum\limits_{p\nmid x}\chi(1-x^{-1})\notag\\
&=\sum\limits_{p\nmid x}\chi(1-x)\notag  \\
&=-p^{\alpha-1}. 
\end{align}
Combining \eqref{e25}, \eqref{e26}, \eqref{e27} and \eqref{ee28}, \eqref{e8} yields 
$$\mathcal{K}_3(G_{p^\alpha})=\dfrac{p^\alpha p^{\alpha-1}(p-1)}{12}\times\dfrac{1}{4} (p-5)p^{\alpha-1},$$
completing the proof of the theorem.
\end{proof}
\begin{proof}[Proof of Theorem \ref{t2}]
We shall proceed in a similar fashion as in  \cite{evans1981number}. Analogous to the beginning of the proof of Theorem \ref{t1}, we have
\begin{equation}\label{e28}
\mathcal{K}_4(G_{p^\alpha})=\dfrac{p^\alpha\times \text{ number of } 4\text{-order cliques containing }0}{4}=\dfrac{p^\alpha \times \mathcal{K}_3(H)}{4}
\end{equation}
 and
 \begin{equation}\label{e29}
 \mathcal{K}_3(H)=\dfrac{p^{\alpha-1}}{3}\left( \dfrac{p-1}{2}\right) \times \text{ number of }3\text{-order cliques in } H \text{ containing } 1.
 \end{equation}
Let $f(p^\alpha)$ be the number of $3$-order cliques in $H$ containing $1$.
Then combining \eqref{e28} and \eqref{e29} we have
\begin{equation}\label{eq1}
\mathcal{K}_4(G_{p^\alpha})=\dfrac{p^{2\alpha-1}(p-1)f(p^\alpha)}{24},
\end{equation}
so we are left to evaluate $f(p^\alpha).$
Let $$X=\{(x,y)\in \mathbb{Z}_{p^\alpha} \times \mathbb{Z}_{p^\alpha} :  p \nmid x,y,1-x^2,1-y^2, x^2-y^2 \} .$$ Then 
$$f(p^\alpha)=\dfrac{1}{8}\times |\{ (x,y)\in X : \chi(1-x^2)=\chi(1-y^2)=\chi(x^2-y^2)=1 \} |.$$
Let $X=A_1\cup A_2\cup \cdots \cup A_8$, where $A_1, A_2, \ldots, A_8$ are as given in the following table.
\begin{table}[h!]
	\begin{center}
		\begin{tabular}{ l | c | c | c }
			Subset of $X$ & $\chi(1-x^2)$ & $\chi(1-y^2)$ & $\chi(x^2-y^2)$\\  
			\hline
			$A_1$ & 1 & 1 & 1\\ 
			$A_2$ & 1 & 1 & $-1$\\ 
			$A_3$ & 1 & $-1$ & 1\\ 
			$A_4$ & 1 & $-1$ & $-1$\\ 
			$A_5$ & $-1$ & 1 & 1\\ 
			$A_6$ & $-1$ & 1 & $-1$\\ 
			$A_7$ & $-1$ & $-1$ & 1\\ 
			$A_8$ & $-1$ & $-1$ & $-1$\\ 
		\end{tabular}
	\end{center}
\end{table}
For example, $A_2$ is the subset of $X$ where $\chi(1-x^2)=1, \chi(1-y^2)=1$ and $\chi(x^2-y^2)=-1$. Let $\beta_i=|A_i|, i=1,2,\ldots ,8.$ Our objective is to find $\beta_1$ since we have $f(p^\alpha)=\frac{1}{8}\beta_1$. To find $\beta_1$, we employ some relations between the $\beta_i$'s which we proceed to find out, similar to what is done in \cite{evans1981number}. Now,
\begin{align*}
&A_1\cup A_2\\
&=\{ (x,y) \in \mathbb{Z}_{p^\alpha}\times \mathbb{Z}_{p^\alpha}: p \nmid x,1-x^2,y,1-y^2, x^2-y^2;\chi(1-x^2)=1;\chi(1-y^2)=1 \}.
\end{align*}
 Using Lemma \ref{5}, we find that
\begin{equation}\label{e30}
 A:=|A_1\cup A_2|=p^{2\alpha-2}\left(  \dfrac{p-5}{2}\right) \left( \dfrac{p-9}{2}\right).
\end{equation}
Again, we have 
\begin{align*}
&A_3\cup A_4\\
&=\{ (x,y) \in \mathbb{Z}_{p^\alpha}\times \mathbb{Z}_{p^\alpha} : p \nmid x,1-x^2,y,1-y^2, x^2-y^2;\chi(1-x^2)=1;\chi(1-y^2)=-1\} .
\end{align*} 
We calculate the cardinality of $A_3\cup A_4$ in the following manner. The total number of $y\in\mathbb{Z}_{p^\alpha}^{\ast}$ such that $\chi(1-y^2)=\pm 1$ is the number of $y\in\mathbb{Z}_{p^\alpha}^{\ast}$ such that $p\nmid y^2-1$, which is $p^\alpha-2p^{\alpha-1}=p^{\alpha-1}(p-2)$. Out of them, the number of $y\in\mathbb{Z}_{p^\alpha}^{\ast}$ such that  $\chi(1-y^2)=1$ is $\frac{p^{\alpha-1}(p-5)}{2}$ by Lemma \ref{5}. So the remaining number of $y\in\mathbb{Z}_{p^\alpha}^{\ast}$ is $p^{\alpha-1}(p-2)-\frac{p^{\alpha-1}(p-5)}{2}- \gamma$, where $\gamma$ is the number of $y$ such that $\chi(1-y^2)=1$  and $p\mid y$. Clearly, $\gamma=p^{\alpha-1}$. Hence 
 \begin{equation}\label{e31}
 B:=|A_3\cup A_4|=p^{2\alpha-2}\left(  \dfrac{p-5}{2}\right) \left( \dfrac{p-1}{2}\right).
 \end{equation}
  Recalling \eqref{e30} and \eqref{e31} we get $$\beta_1+\beta_2=A$$ 
$$\beta_3+\beta_4=B.$$
We define a bijection from $A_1\cup A_3$ to $A_1\cup A_2$ as
$$(x,y)\mapsto (x,x y^{-1}).$$
Then $\beta_1+\beta_3=\beta_1+\beta_2$. In a similar fashion we get other relations among the $\beta_i$'s in terms of $A$ and $B$ (as named in \eqref{e30} and \eqref{e31}). In particular, we have
\begin{align}\label{e45}
\beta_1+\beta_2=A;\notag\\
\beta_1+\beta_3=A;\notag\\
\beta_3+\beta_4=B;\notag\\
\beta_1+\beta_5=A;\notag\\
\beta_2+\beta_6=B;\notag\\
\beta_5+\beta_7=B;\notag\\
\beta_7+\beta_8=B.
\end{align}
Let $$S=\sum_{x\in\mathbb{Z}_{p^\alpha}}\sum_{y\in\mathbb{Z}_{p^\alpha}} \chi((1-x^2)(1-y^2)(x^2-y^2))$$ and 
$$S_0=\mathop{\sum \sum}_{\substack{(x,y)\in X}} \chi((1-x^2)(1-y^2)(x^2-y^2)).$$
Using \eqref{e45} we have
\begin{align}\label{e32}
S_0&=\mathop{\sum \sum}_{\substack{(x,y)\in A_1}}\chi((1-x^2)(1-y^2)(x^2-y^2))+\cdots+\mathop{\sum \sum}_{\substack{(x,y)\in A_8}}\chi((1-x^2)(1-y^2)(x^2-y^2))\notag\\
&=\beta_1-\beta_2-\beta_3+\beta_4-\beta_5+\beta_6+\beta_7-\beta_8\notag\\
&=\beta_1-(A-\beta_1)-(A-\beta_1)+(B-\beta_3)-(A-\beta_1)+(B-\beta_2)+(B-\beta_5)-(B-\beta_7)\notag\\
&=64 f(p^\alpha)+p^{2\alpha-2}(p-5)(15-p).
\end{align}
Using Lemma \ref{9} and \eqref{e32} we have
\begin{equation}\label{eq2}
f(p^\alpha)=\dfrac{S+p^{2\alpha-2}(p^2-20p+79)}{64}.
\end{equation}
Now, \begin{align}\label{e9}
S&=\sum_{p\mid x}\sum_{y}\chi((1-x^2)(1-y^2)(x^2-y^2))+\sum_{p\nmid x}\sum_{y}\chi((1-x^2)(1-y^2)(x^2-y^2)).
\end{align}
The first term in \eqref{e9} is
\begin{align*}
&\sum_{p\mid x}\sum_y \chi(1-x^2)\chi((1-y^2)(x^2-y^2))\\
=& \sum_{p\mid x}\sum_{p\mid y} \chi(1-y^2)\chi(x^2-y^2)+\sum_{p\mid x}\sum_{p\nmid y} \chi(1-y^2)\chi(x^2-y^2)\\
=& \sum_{p\mid x}\sum_{p\nmid y} \chi(1-y^2)\chi(x^2-y^2) \\
=&\sum_{p\mid x}\sum_{p\nmid y} \chi(1-y^2)\\
=&-2p^{2\alpha-2}.
\end{align*} So \eqref{e9} yields 
\begin{align}\label{e33}
S&=-2p^{2\alpha-2}+\sum_{p\nmid x}\sum_{ y} \chi((1-x^2)(1-y^2)(x^2-y^2))\notag\\
&=-2p^{2\alpha-2}+\sum_{p\nmid x}\sum_{p\mid y}\chi((1-x^2)(1-y^2)(x^2-y^2))\notag\\
&+\sum_{p\nmid x}\sum_{p\nmid y}\chi((1-x^2)(1-y^2)(x^2-y^2)).
\end{align}
We have
$$\sum_{p\nmid x}\sum_{p\mid y} \chi(1-x^2)\chi(1-y^2)\chi(x^2-y^2)=\sum_{p\nmid x}\sum_{p\mid y} \chi(1-x^2)=-2p^{2\alpha-2},$$ so \eqref{e33} yields
\begin{align*}
S&=-4p^{2\alpha-2}+\sum_{p\nmid x}\sum_{p\nmid y}\chi((1-x^2)(1-y^2)(x^2-y^2))\\
&=-4p^{2\alpha-2}+\sum_{p\nmid x}\chi(1-x^2)\sum_{p\nmid y}\chi((1-y)(x^2-y))\left\lbrace 1+\chi(y)\right\rbrace\\
&= -4p^{2\alpha-2}+\sum_{p\nmid y}\chi(1-y)\left\lbrace 1+\chi(y)\right\rbrace \sum_{p\nmid x}\chi((1-x^2)(y-x^2))\\
&=-4p^{2\alpha-2}+\sum_{p\nmid y}\chi(1-y)\left\lbrace 1+\chi(y)\right\rbrace \sum_{p\nmid x}\chi((1-x)(y-x))\left\lbrace 1+\chi(x)\right\rbrace  \\
&=-4p^{2\alpha-2}+\sum_{p\nmid x} \sum_{p\nmid y}\chi((1-x)(1-y)(y-x))\left\lbrace 1+\chi(x)\right\rbrace\left\lbrace 1+\chi(y)\right\rbrace.
\end{align*}
Now, we break the sum $S$ into three parts as in \cite{evans1981number}. Let
 \begin{equation}\label{e10}
S=-4p^{2\alpha-2}+I+2J+K,
\end{equation}
where
$$I=\sum_{p\nmid x} \sum_{p\nmid y}\chi((1-x)(1-y)(y-x)),$$
$$J=\sum_{p\nmid x} \sum_{p\nmid y}\chi((1-x)(1-y)(y-x)x),$$ and
$$K=\sum_{p\nmid x} \sum_{p\nmid y}\chi((1-x)(1-y)(y-x)xy).$$
We evaluate $I$ using Lemma \ref{6}.
\begin{align}\label{e34}
I&=\sum_{p\nmid x}\sum_{p\nmid y}\chi((1-x)(1-y)(y-x))\notag \\
&=\sum_{p\nmid x}\chi(1-x)\sum_{p\nmid y}\chi((1-y)(x-y))\notag\\
&=\sum_{p\nmid x}\chi(1-x)\left[\sum_{y}\chi((1-y)(x-y))- \sum_{p\mid y}\chi((1-y)(x-y))\right]\notag \\
&=\sum_{\substack{p\nmid x\\p\nmid 1-x}}\chi(1-x)\left[ -p^{\alpha-1}-p^{\alpha-1}\chi(x)\right] \notag\\
&=2p^{2\alpha-2}.
\end{align}
We further evaluate $J$ and evoke Lemma \ref{6} for the same.
\begin{align}\label{e35}
J&=\sum_{p\nmid x}\sum_{p\nmid y}\chi((1-x)(1-y)(y-x)x)\notag \\
&=\sum_{p\nmid x,x-1}\chi(x(1-x))\sum_{p\nmid y}\chi((1-y)(y-x))\notag \\
&=\sum_{p\nmid x,x-1}\chi(x(1-x))\left[ -p^{\alpha-1}-p^{\alpha-1}\chi(x)\right] \notag \\
&=2 p^{2\alpha-2}.
\end{align}
Using \eqref{e34}, \eqref{e35} and Lemma \ref{10}, \eqref{e10} yields
\begin{equation}\label{e36}
 S=-4 p^{2\alpha-2}+I+2J+K=2p^{2\alpha-2}+J(\psi,\chi)^2+\overline{J(\psi,\chi)}^2.
\end{equation}
Combining \eqref{eq2} and \eqref{e36}, we obtain $$f(p^\alpha)=\dfrac{p^{2\alpha-2}(p^2-20p+81)+J(\psi,\chi)^2+\overline{J(\psi,\chi)}^2}{64}.$$
Finally putting the value of $f(p^{\alpha})$ in \eqref{eq1}, we complete the proof.
\end{proof}


\begin{thebibliography}{99}
	\bibitem{ananchuen2001adjacency}
	W. Ananchuen, {\it On the adjacency properties of generalized Paley graphs}, Australasian Journal of Combinatorics 24 (2001), 129--148. 
	
	\bibitem{ananchuen1993adjacency}
	W. Ananchuen and L. Caccetta, {\it On the adjacency properties of Paley graphs}, Networks 23 (1993), 227--236.
	
	\bibitem{ananchuen2006cubic}
	W. Ananchuen and L. Caccetta, {\it Cubic and quadruple Paley graphs with the ne. c. property}, Discrete Mathematics 306 (2006), 2954--2961.
	
	\bibitem{ashrafi2010unit}
	N. Ashrafi and H.R. Maimani and M.R. Pournaki and S. Yassemi, {\it Unit graphs associated with rings}, Communications in Algebra 38 (2010), 2851--2871. 
	
	\bibitem{atanasov2014certain}
	R. Atanasov and M. Budden and J. Lambert and  K. Murphy and A. Penland, {\it On certain induced subgraphs of Paley graphs}, Acta Univ. Apulensis Math. Inform. 40 (2014), 51--65.
	
	\bibitem{berndt1979sums}
	B. C Berndt and R. J Evans, {\it Sums of Gauss, Jacobi, and Jacobsthal}, Journal of Number Theory 11 (1979), 349--398.
	
	\bibitem{dawsey2020generalized}
	M.L. Dawsey and D. McCarthy, {\it Generalized Paley graphs and their complete subgraphs of orders three and four}, Res Math Sci 8, Article number 18 (2021).
	
	
	\bibitem{elsawy2012paley}
	A.N. Elsawy, {\it Paley graphs and their generalizations}, arXiv preprint arXiv:1203.1818 (2012).
	
	\bibitem{erdos1962number}
	P. Erd$\ddot{\text{o}}$s, {\it On the number of complete subgraphs contained in certain graphs},  Magyar Tud. Akad. Mat. Kutat{\'o} Int. K{\"o}zl 7 (1962), 459--464.
	
	\bibitem{evans1981number}
	R.J. Evans and J.R. Pulham and J. Sheehan, {\it On the number of complete subgraphs contained in certain graphs}, Journal of Combinatorial Theory, Series B, 3 (1981), 364--371.
	
	\bibitem{goodman1959sets}
	A. W. Goodman, {\it On sets of acquaintances and strangers at any party}, The American Mathematical Monthly 66 (1959), 778--783.
	
	\bibitem{ireland1990classical}
	K. Ireland and M. Rosen, {\it A Classical Introduction to Modern Number Theory}, Grad. Texts in Math (1990).
	
	\bibitem{jones2017paley}
	G. A. Jones, {\it Paley and the Paley graphs}, arXiv preprint arXiv:1702.00285 (2017).
	
	\bibitem{lim2006generalised}
	T. K. Lim and C. E. Praeger, {\it On generalised Paley graphs and their automorphism groups}, Michigan Mathematical Journal 58 (2009), 293--308. 
	
	
	\bibitem{paley1933orthogonal}
	R. E. A. C. Paley, {\it On orthogonal matrices}, Journal of Mathematics and Physics 12 (1933), 311--320. 
	
	\bibitem{thomason}
	A. G. Thomason, {\it Partitions of Graphs}, Ph.D. thesis, Cambridge University (1979).
	
	\bibitem{wage2006character}
	N. Wage, {\it Character sums and Ramsey properties of generalized Paley graphs}, Integers 6 (2006), article no. 18.
	
	\end{thebibliography}

\section{Acknowledgement}
We are grateful to the referees for many helpful comments.
\section{Declarations}
\begin{itemize}
	\item Funding: Not applicable.
	\item Data availability: Not applicable.
	\item Code availability: Not applicable.
	\item Conflict of interest: The authors declare that there is no conflict of interest.
\end{itemize}

\end{document}